\documentclass[11pt]{article}%
\usepackage{amsmath,mathrsfs}
\usepackage{amsfonts}
\usepackage{amssymb}
\usepackage{graphicx,subfigure,color}
\usepackage{algorithm}
\usepackage{multirow}
\usepackage{xcolor}
\usepackage[affil-it]{authblk}
\usepackage{setspace}
\usepackage[numbers,sort&compress]{natbib}
\usepackage{subfigure}
\usepackage{booktabs}
\usepackage{latexsym,array,graphicx}
\providecommand{\U}[1]{\protect\rule{.1in}{.1in}}
\allowdisplaybreaks[4]

\usepackage{algorithm,algpseudocode,float}



\numberwithin{equation}{section}
\newenvironment{keywords}{\noindent{\bf Key words:}}

\newenvironment{AMS}{\noindent{\bf Mathematics Subject Classification:}}

\newtheorem{theorem}{Theorem}
\newtheorem{example}{Example}
\newtheorem{lemma}[theorem]{Lemma}

\newtheorem{assumption}{Assumption}
\newtheorem{proposition}{Proposition}
\newtheorem{remark}{Remark}

\newtheorem{corollary}[theorem]{Corollary}
\newcommand{\diag}{\mathrm{diag}}
\newcommand{\blockdiag}{\mathrm{blkdiag}}

\newenvironment{proof}{\noindent{\bf Proof:}}{\hfill\fbox{}\vspace*{1mm}}

\title{A parallel-in-time two-sided preconditioning for all-at-once system from a non-local evolutionary equation with weakly singular kernel }
\author[a,b]{Xue-lei Lin \thanks{Corresponding author. e-mail: linxuelei@csrc.ac.cn}}
\author[c]{Michael K. Ng \thanks{e-mail: mng@maths.hku.hk}}
\author[d]{Yajing Zhi \thanks{e-mail: yjzhi@connect.hku.hk}}

\affil[a]{
Shenzhen JL Computational Science and Applied Research Institute, Shenzhen, P.R. China.
}
\affil[b]{Beijing Computational Science Research Center, Beijing 100193, China.
}
\affil[c]{%
	Department of Mathematics, The University of Hong Kong, Pokfulam, Hong Kong.}
\affil[d]{%
	Department of Computer Science, The University of Hong Kong, Pokfulam, Hong Kong.}
\date{}
\begin{document}	
	\maketitle

\begin{abstract}
In this paper, we study \textcolor{black}{a} parallel-in-time (PinT) algorithm for all-at-once system from
a non-local evolutionary equation with weakly singular kernel where the temporal term involves a non-local convolution with a weakly singular kernel and the spatial term is the usual Laplacian operator with variable coefficients. Such a problem has been intensively studied in recent years thanks to the \textcolor{black}{numerous} real world applications. However, due to the non-local property of the time evolution, solving the equation in PinT manner is difficult. We propose to use a two-sided preconditioning technique for the all-at-once discretization of the equation. Our preconditioner is constructed by replacing the variable diffusion coefficients with a constant coefficient to obtain a constant-coefficient all-at-once matrix. We split a square root of \textcolor{black}{the} constant Laplacian operator out of the constant-coefficient all-at-once matrix as a right preconditioner and take the remaining part as a left preconditioner, which constitutes our two-sided preconditioning. Exploiting the diagonalizability of the constant-Laplacian matrix and the triangular Toeplitz structure of the temporal discretization matrix, we obtain efficient representations  of inverses of \textcolor{black}{the} right and \textcolor{black}{the} left preconditioners, because of which the iterative solution can be fast updated in \textcolor{black}{a} PinT manner.
Theoretically, the condition number of \textcolor{black}{the} two-sided preconditioned matrix is proven to be uniformly bounded by a constant independent of \textcolor{black}{the}  matrix size. To the best of our knowledge, for the non-local evolutionary equation with variable coefficients, this is the first attempt to develop a PinT preconditioning technique that has fast and exact implementation and that the corresponding preconditioned system has a uniformly bounded condition number.
 Numerical results are reported to confirm the 
efficiency of the proposed two-sided preconditioning technique.
\end{abstract}


	\section{Introduction}\label{introduction}
Consider a non-local evolutionary equation with \textcolor{black}{a} weakly singular kernel: 
\begin{align}
&\frac{1}{\Gamma(1-\alpha)}\int_{0}^{t}\frac{\partial u({\bf x},s)}{\partial s}(t-s)^{-\alpha}ds=\nabla{\boldsymbol{\cdot}}(a({\bf x})\nabla u)+f({\bf x},t),\quad {\bf x}\in \Omega\subset\mathbb{R}^d,~ t\in(0,T],\label{sub-diffusioneq}\\
& u({\bf x},t)=0,\quad {\bf x}\in\partial \Omega,~t\in(0,T],\label{Drcheltboundary}\\
& u({\bf x},0)=\psi({\bf x}),\quad {\bf x}\in \Omega,\label{initial condition}
\end{align}
where $\Gamma(\cdot)$ is the gamma function, $\alpha\in(0,1)$, $\Omega=\prod_{i=1}^{d}(\check{c}_i,\hat{c}_i)$ is an open hyper-rectangle; $\partial\Omega$ denotes the boundary of $\Omega$; $a({\bf x})\in[\check{a},\hat{a}]$  with some positive constants $0<\check{a}\leq\hat{a}$; $a$, $f$ and $\psi$ are all known functions.

Let ${\bf L}_a\in\mathbb{R}^{J\times J}$ be the central finite difference discretization of $-\nabla(a({\bf x})\nabla)$ on uniform grid. For \textcolor{black}{the} temporal discretization, we focus on a convolution quadrature, L1 scheme \cite{sun2006fully,lin2007finite,jin2016analysis,liao2018sharp} in this paper. With the L1 scheme, the temporal discretization has the following form
\begin{align}
\frac{1}{\Gamma(1-\alpha)}\int_{0}^{n\tau}\frac{\partial u({\bf x},s)}{\partial s}(t-s)^{-\alpha}ds\approx \frac{1}{\tau^{\alpha}}\sum\limits_{k=1}^{n}l_{n-k}^{(\alpha)}u({\bf x},n\tau)+\frac{1}{\tau^{\alpha}}l^{n,\alpha}\psi({\bf x}),~{\bf x}\in\Omega,~ n=1,2,..., N,\label{convlqdra}
\end{align}
 where $\tau=T/N$ is the temporal step-size, $N$ is the total number of time steps. There are also other convolution quadratures fitting the form \eqref{convlqdra}; see, e.g.,  Gr$\ddot{\rm u}$nwald formula \cite[pp. 208]{podlubny1999},  L1-2 scheme \cite{gao2014new}, L2-$1_{\sigma}$ scheme \cite{alikhanov2015new}. 
 
Combining \eqref{convlqdra} and the finite difference spatial discretization, we obtain a full discretization of \eqref{sub-diffusioneq}--\eqref{initial condition} as follows
\begin{equation}\label{timestpdisc}
\frac{1}{\tau^\alpha}\sum\limits_{k=1}^{n}l_{n-k}^{(\alpha)}\tilde{\bf u}^k+{\bf L}_a\tilde{\bf u}^n=\tilde{\bf f}^n,\quad  n=1,2,...,N,
\end{equation}
where $\tilde{\bf u}^n\in\mathbb{R}^{J\times 1}$ is a vector whose components are approximate values of $u(\cdot,n\tau)$ on spatial grid points arranged in \textcolor{black}{a} lexicographic ordering, $\tilde{\bf f}^n$ contains the initial condition and the values of $f(\cdot,n\tau)$ on \textcolor{black}{the} spatial grid points. 
For \textcolor{black}{the} stability and convergence proof of the full discretization scheme \eqref{timestpdisc}, one may refer to \cite{zhang2011alternating,jin2016analysis,liao2018sharp}.

Putting the $N$ many linear systems in a single large linear system, we obtain the all-at-once \textcolor{black}{system} as follows
\begin{align}\label{allatoncesysts}
\tilde{\bf A}\tilde{\bf u}=\tilde{\bf f},
\end{align}
where $\tilde{\bf u}=(\tilde{\bf u}^{1};\tilde{\bf u}^{2};\cdots;\tilde{\bf u}^{N})\in\mathbb{R}^{NJ\times 1}$, $\tilde{\bf f}=(\tilde{\bf f}^{1};\tilde{\bf f}^{2};\cdots;\tilde{\bf f}^{N})\in\mathbb{R}^{NJ\times 1}$,
\begin{align*}
	\tilde{\bf A}={\bf I}_{N}\otimes{\bf L}_a+{\bf T}\otimes{\bf I}_{ J },\quad {\bf T}=\frac{1}{\tau^\alpha}\left[
	\begin{array}
		[c]{cccc}
		l_{0}^{(\alpha)} &       &  &\\
		l_{1}^{( \alpha)}& l_{0}^{(\alpha)} &  &\\
		\vdots&\ddots &\ddots&\\
		l_{N-1}^{( \alpha)}& \ldots & l_{1}^{( \alpha)} & l_{0}^{(\alpha)}
	\end{array}
	\right],
\end{align*}
${\bf I}_k$ denotes \textcolor{black}{the} $k\times k$ identity matrix, `$\otimes$' denotes the Kronecker product.


The non-local evolutionary equation with \textcolor{black}{the} weakly singular kernel has attracted much attention in recent years, thanks to its \textcolor{black}{numerous} real world applications. The classical local evolutionary equation (with \textcolor{black}{the} first order temporal derivative) rests on the assumption that the mean square particle displacement grows linearly with respect to time. A lot of experimental studies indicate that the linear-growth assumption may not be accurate enough to describe some physical processes in which the mean square displacement grows sub-linearly or super-linearly with respect to time. These experimental studies \textcolor{black}{cover} a wide range of important practical applications including  visco-elastic materials \cite{caputo1966Linear,eidulman2004},  thermal diffusion in fractal domain \cite{nigmatullin2010the}, column experiments \cite{hatano1998}, protein transport in cell membrane \cite{kou2008}.
Actually, if the underlying stochastic process is defined by continuous time random walk, then the non-local evolutionary equation \eqref{sub-diffusioneq}--\eqref{initial condition} is exactly a macroscopic model for the probability density function of particles whose mean square displacement grows sub-linearly like $t^{\alpha}$. For more applications of the non-local evolutionary equation, we refer interested readers to \textcolor{black}{the} survey papers \cite{2019Numerical,sun2018213,metzler2014anomalous,metzler2000random}.

The time-stepping method solves $\tilde{\bf u}^{n+1}$ \eqref{timestpdisc} after solving $\tilde{\bf u}^{n}$ , which is a sequential solver. Moreover, due to the non-local time evolution in \eqref{sub-diffusioneq}, $l_k^{(\alpha)}$'s $(k=1,2,...,N)$ are all nonzero numbers. Because of this, before solving the linear system $(\frac{1}{\tau^{\alpha}}l_0^{(\alpha)}{\bf I}+{\bf L}_a)\tilde{\bf u}^n=\tilde{\bf f}^{n}-\frac{1}{\tau^{\alpha}}\sum\limits_{k=1}^{n-1}l_{n-k}^{(\alpha)}\tilde{\bf u}^{k}$, one has to compute the non-local summation $-\frac{1}{\tau^{\alpha}}\sum\limits_{k=1}^{n-1}l_{n-k}^{(\alpha)}\tilde{\bf u}^{k}$ over all previous time steps. \textcolor{black}{Computing these non-local summation requires $\mathcal{O}(N^2J)$ flops in total. The computation of such non-local summation is already expensive, let alone solving \textcolor{black}{those} local time linear systems.} To remedy the situation, fast kernel compression methods are proposed in \cite{jiang2017fast,mclean2012fast,li2010fast,baffet2017kernel} so that the non-local summation over  \textcolor{black}{all the} previous time steps can be reduced to a local summation over only a few previous time steps. With these kernel compression techniques, \textcolor{black}{the operation cost for computing the right hand side vectors for those local time linear systems can be reduced to $\mathcal{O}(N\log^2 N)$ flops.} However, these kernel compression methods are still implemented in a time-stepping pattern.

\textcolor{black}{Besides the time-stepping method, PinT method is another type of popular methods for \eqref{timestpdisc}, which solves  $\{\tilde{\bf u}^{n}|n=1,2,...,N\}$ in a parallel way.}  In \cite{lin2016fast,gu2020109576,lupangsun,lu2017Approximate}, PinT algorithms are developed by employing a fast diagonalizable approximation to ${\bf T}$ in the all-at-once system \eqref{allatoncesysts}. The so approximated all-at-once system is block diagonalizable with each eigen-block corresponding to a complex scalar shifted spatial system. When $a$ is a constant or $d=1$, these complex spatial systems can be solved by fast Poisson solver or fast banded solver. However, when $a$ is not a constant and $d>1$, then there is no fast direct solver for these complex spatial systems. In \cite{lin2016fast}, multigrid iterative solvers are proposed to solve these complex spatial systems. In general, there is no theoretical convergence guarantee for these complex multigrid solvers. 

There are also other PinT algorithms proposed for the non-local evolutionary equation \eqref{sub-diffusioneq}--\eqref{initial condition} rather than the full discrete equations \eqref{timestpdisc}. These algorithms include the temporal Laplace transform based algorithms \cite{kwon2003parallel,mclean2010maximum} and 
parareal algorithms \cite{xu2015parareal,wu2018parareal,li2013parallel,fu2019preconditioned}. The Laplace transform based algorithms highly depends on regularity of $f$ and it requires the computation of \textcolor{black}{the} temporal Laplace transform of $f$, which is usually expensive; see, e.g., \cite{gu2020109576,sheen2003parallel,sheen2000parallel,wu2017laplace} for more discussions. For local time problems, the parareal algorithm was firstly proposed in \cite{lions2001} and an interesting improved version using \textcolor{black}{the} parallel coarse grid correction can be found in \cite{wushulin2018}. Due to the non-local temporal convolution, a direct application of the parareal algorithm to the non-local evolutionary problem  has its computational time non-uniformly distributed among the processors, because of which the \textcolor{black}{processors for earlier time-steps have} longer time waiting phenomenon. In \cite{wu2018parareal}, a local time-integrator based \textcolor{black}{the} parareal algorithm is proposed for the non-local evolutionary equation to avoid the time waiting phenomenon. However, numerical results in \cite{wu2018solving} indicates that the local time-integrator based \textcolor{black}{the} parareal algorithm for the non-local evolutionary equation may converge slowly or even diverge sometimes.

By exchanging the order of the Kronecker product in $\tilde{\bf A}$, the all-at-once system \eqref{allatoncesysts} can be equivalently rewritten as
\begin{equation}\label{allatoncesysst}
{\bf A}{\bf u}={\bf f},
\end{equation}
where ${\bf A}={\bf L}_a\otimes{\bf I}_N+{\bf I}_J\otimes {\bf T}$. Clearly, \eqref{allatoncesysst} can be obtained by applying \textcolor{black}{a simple permutation transformation} to $\tilde{\bf u}$ and $\tilde{\bf f}$.

In this paper, we propose a novel two-sided PinT preconditioning technique for the all-at-once system  \eqref{allatoncesysst} with non-constant $a({\bf x})$ and arbitrary $d$. Our preconditioning technique begins with replacing $a({\bf x})$ by a constant $\beta$ to obtain the constant-coefficient all-at-once matrix ${\bf P}:=(\beta{\bf L}_1)\otimes{\bf I}_N+{\bf I}_J\otimes {\bf T}$, where ${\bf L}_1\in\mathbb{R}^{J\times J}$ is the discretization matrix of \textcolor{black}{the} constant Laplacian $-\nabla^2$.
Instead of applying ${\bf P}$ directly, we develop a two-sided preconditioning technique from ${\bf P}$. The right preconditioner is ${\bf P}_r:=(\beta{\bf L}_1)^{\frac{1}{2}}\otimes{\bf I}_N$ and the left preconditioner ${\bf P}_l:={\bf P}{\bf P}_r^{-1}=(\beta{\bf L}_1)^{\frac{1}{2}}\otimes{\bf I}_N+(\beta{\bf L}_1)^{-\frac{1}{2}}\times {\bf T}$. It is proven in Theorem \ref{mainthm} that the condition number of the two-sided preconditioned matrix ${\bf P}_l^{-1}{\bf A}{\bf P}_r^{-1}$ is uniformly bounded by the constant $\hat{a}/\check{a}$, which indicates that the convergence rate of  Krylov subspace solver for the two-sided preconditioned system  does not deteriorate as $N$ or $J$ increases. Indeed, the numerical results in Section \ref{experimentsection} show that the iteration number for the two-sided preconditioned system keeps bounded as $N$ or $J$ increases. 
Thanks to \textcolor{black}{the} fast diagonalizability of ${\bf L}_1$ and triangular Toeplitz structure of ${\bf T}$, ${\bf P}_r^{-1}$ is diagonalizable by multi-dimension fast sine transform (FST) and ${\bf P}_l^{-1}$ is block diagonalizable by multi-dimension FST with each eigen-block being a triangular Toeplitz matrix. Because of the efficient representations of ${\bf P}_r^{-1}$ and ${\bf P}_l^{-1}$, the matrix-vector product ${\bf P}_l^{-1}{\bf A}{\bf P}_r^{-1}{\bf v}$ for a  given vector ${\bf v}$ can be computed in a PinT pattern, requiring $\mathcal{O}(NJ\log (NJ))$ flops; see Section \ref{implementsection} for more details of \textcolor{black}{the} implementation. Hence, solving the all-at-once system using our two-sided preconditioning iterative method requires $\mathcal{O}(NJ\log(NJ))$ flops in total, which is nearly optimal. The study of PinT preconditioning techniques for \textcolor{black}{the} all-at-once system from the non-local evolutionary equation is still in its infancy.
To the best of our knowledge, for the non-local evolutionary equation \eqref{sub-diffusioneq}--\eqref{initial condition} with non-constant $a$, this is the first attempt to develop a PinT preconditioning technique that can be fast and exactly implemented in a linearithmic complexity and that the corresponding preconditioned system has a uniformly bounded condition number independent of $N$ and $J$.

The rest of this paper is organized as follows. 
In Section 2, the condition number of the two-sided preconditioned all-at-once  matrix is analyzed.
In Section 3, a fast implementation for \textcolor{black}{the} two-sided preconditioning method is proposed and
its complexity is discussed.
In Section \ref{implementsection}, numerical results are reported. In Section 5, we give conclusions and discusses few issues to be considered in future works.

\section{A Two-Sided Preconditioning and Condition Number of the Preconditioned Matrix} \label{mainrsltsec}
In this section, we introduce the two-sided preconditioners and estimate \textcolor{black}{the} condition number of the preconditioned matrix. 

Denote
\begin{equation*}
\beta=\sqrt{\hat{a}\check{a}}.
\end{equation*}
It is clear that $\check{a}\leq\beta\leq\hat{a}$.

For any real symmetric positive semi-definite matrix ${\bf B}\in\mathbb{R}^{k\times k}$, define $${\bf B}^{z}:={\bf V}^{\rm T}\diag([{\bf D}(i,i)]^z)_{i=1}^{k}{\bf V},\quad z\in\mathbb{R},$$ where ${\bf B}={\bf V}^{\rm T}{\bf D}{\bf V}$ denotes the orthogonal diagonalization of ${\bf B}$.

Recall that ${\bf L}_{1}$ denotes the discretization of the constant-coefficient Laplacian $-\nabla^2$.

As mentioned in Section \ref{introduction}, our left preconditioner ${\bf P}_l$ and right preconditioner ${\bf P}_r$ are defined as follows
\begin{equation*}
{\bf P}_l=(\beta{\bf L}_{1})^{-\frac{1}{2}}\otimes{\bf T}+(\beta{\bf L}_{1})^{\frac{1}{2}}\otimes{\bf I}_N,\quad {\bf P}_r=(\beta{\bf L}_{1})^{\frac{1}{2}}\otimes{\bf I}_N.
\end{equation*}
Then, \textcolor{black}{to solve \eqref{allatoncesysst}, it is equivalent to solve the linear system \eqref{twosidedpreclinsys} and to compute the scaling step \eqref{scalling}}
\begin{align}
&{\bf P}_l^{-1}{\bf A}{\bf P}_r^{-1}\hat{\bf u}={\bf P}_l^{-1}{\bf f},\label{twosidedpreclinsys},\\
&{\bf u}={\bf P}_r^{-1}\hat{\bf u}.\label{scalling}
\end{align}
\eqref{twosidedpreclinsys} is the so-called two-sided preconditioned linear system. \textcolor{black}{The} Krylov subspace solver is employed to solve \eqref{twosidedpreclinsys}. In iteration process of a Krylov subspace solver for \eqref{scalling}, it only requires to compute some matrix-vector products ${\bf P}_l^{-1}{\bf A}{\bf P}_r^{-1}{\bf v}$ for some given vectors ${\bf v}$. Hence, the matrix ${\bf P}_l^{-1}{\bf A}{\bf P}_r^{-1}$ is never formed explicitly in the computational process. Instead, ${\bf P}_l^{-1}{\bf A}{\bf P}_r^{-1}{\bf v}$ is computed by $({\bf P}_l^{-1}({\bf A}({\bf P}_r^{-1}{\bf v})))$. 
More details on fast computation of the matrix-vector product and \eqref{scalling} will be discussed in Section \ref{implementsection}.

Before estimating the condition number of the preconditioned matrix ${\bf P}_l^{-1}{\bf A}{\bf P}_r^{-1}$, we introduce some preliminaries and lemmas first.

For any real symmetric matrices ${\bf C}_1,{\bf C}_2\in\mathbb{R}^{k\times k}$, denote ${\bf C}_2 \succ ({\rm or} \succeq) \ {\bf C}_1$
if ${\bf C}_2-{\bf C}_1$ is positive definite (or semi-definite). Especially, we denote ${\bf C}_2 \succ ({\rm or} \succeq) \ {\bf O}$, if ${\bf C}_2$ itself is positive definite (or semi-definite).
Also, ${\bf C}_1 \prec ({\rm or} \preceq) \ {\bf C}_2$ and ${\bf O} \prec ({\rm or} \preceq) \ {\bf C}_2$  have the same meanings as those of ${\bf C}_2 \succ ({\rm or} \succeq) \ {\bf C}_1$ and ${\bf C}_2 \succ ({\rm or} \succeq) \ {\bf O}$, respectively.

Some assumptions on the spatial discretization matrix ${\bf L}_{a}$ are listed as follows.
\textcolor{black}{\begin{assumption}\label{diclaplaprop}
	\begin{description}
		\item[~]
		\item[(i)]${\bf L}_{a}$ is linear with respect to $a({\bf x})$, i.e., for any real-valued function $p$ and $q$ defined on $\Omega$ and any real constant $c$, it holds ${\bf L}_{p+q}={\bf L}_{p}+{\bf L}_{q}$ and ${\bf L}_{cp}=c{\bf L}_{p}$.
		\item[(ii)]${\bf L}_{p}$ is symmetric for any $p:\Omega\mapsto\mathbb{R}$.
		\item[(iii)]${\bf L}_{p}\succeq{\bf O}$ for any nonnegative function $p:\Omega\mapsto\mathbb{R}$.
		\item[(iv)] ${\bf L}_{p}\succ{\bf O}$ for any positive function $p:\Omega\mapsto\mathbb{R}$.
	\end{description}
\end{assumption}}

\begin{lemma}\label{barlineqs}
	\begin{description}
		\item[(i)]$ {\bf O}\prec\check{a}{\bf L}_{1}\preceq {\bf L}_a\preceq\hat{a}{\bf L}_{1}$;
		\item [(ii)] $\check{a}{\bf I}_{J}\preceq{\bf L}_{1}^{-\frac{1}{2}}{\bf L}_a{\bf L}_{1}^{-\frac{1}{2}}\preceq \hat{a}{\bf I}_{J}$.
	\end{description}
	where we recall that the positive constants $\check{a}$ and $\hat{a}$ are lower and upper bounds of $a$, respectively.
\end{lemma}
\begin{proof}
	${\bf (i)}$ follows from \textcolor{black}{Assumption \ref{diclaplaprop}}. 
	
	Let ${\bf z}\in\mathbb{R}^{J\times 1}$ be an arbitrary nonzero vector. Then,
	\begin{equation*}
	\frac{{\bf z}^{\rm T}{\bf L}_{1}^{-\frac{1}{2}}{\bf L}_a{\bf L}_{1}^{-\frac{1}{2}}{\bf z}}{{\bf z}^{\rm T}{\bf z}}\stackrel{{\bf y}={\bf L}_{1}^{-\frac{1}{2}}{\bf z}}{=\joinrel=\joinrel=\joinrel=\joinrel=\joinrel=}\frac{{\bf y}^{\rm T}{\bf L}_a{\bf y}}{{\bf y}^{\rm T}{\bf L}_{1}{\bf y}},
	\end{equation*}
	which together with ${\bf (i)}$ implies that
	\begin{equation*}
	\check{a}\leq\frac{{\bf z}^{\rm T}{\bf L}_{1}^{-\frac{1}{2}}{\bf L}_a{\bf L}_{1}^{-\frac{1}{2}}{\bf z}}{{\bf z}^{\rm T}{\bf z}}\leq\hat{a}.
	\end{equation*}
	The proof is complete.
\end{proof}

\begin{lemma}\label{spdneqs}
	Let ${\bf B}_1,{\bf B}_2\in\mathbb{R}^{k\times k}$ be real symmetric \textcolor{black}{matrices} such that ${\bf O}\prec{\bf B}_1\preceq{\bf B}_2$. Then, ${\bf O}\prec{\bf B}_2^{-1}\preceq{\bf B}_1^{-1}.$
\end{lemma}
\begin{proof}	
	It is clear that ${\bf B}_1^{-1}\succ{\bf O}$ and ${\bf B}_2^{-1}\succ{\bf O}$.
	Notice that ${\bf B}_1^{\frac{1}{2}}{\bf B}_2^{-1}{\bf B}_1^{\frac{1}{2}}=({\bf B}_1^{\frac{1}{2}}{\bf B}_2^{-\frac{1}{2}})({\bf B}_2^{-\frac{1}{2}}{\bf B}_1^{\frac{1}{2}})$ is similar\footnote{see the definition of matrix similarity in \cite{horn2012matrix}} to $({\bf B}_2^{-\frac{1}{2}}{\bf B}_1^{\frac{1}{2}})({\bf B}_1^{\frac{1}{2}}{\bf B}_2^{-\frac{1}{2}})={\bf B}_2^{-\frac{1}{2}}{\bf B}_1{\bf B}_2^{-\frac{1}{2}}$, which means ${\bf B}_1^{\frac{1}{2}}{\bf B}_2^{-1}{\bf B}_1^{\frac{1}{2}}$ and ${\bf B}_2^{-\frac{1}{2}}{\bf B}_1{\bf B}_2^{-\frac{1}{2}}$ have the same spectrum.  Moreover,
	\begin{align*}
	\frac{{\bf z}^{\rm T}{\bf B}_2^{-\frac{1}{2}}{\bf B}_1{\bf B}_2^{-\frac{1}{2}}{\bf z}}{{\bf z}^{\rm T}{\bf z}}\stackrel{{\bf y}={\bf B}_2^{-\frac{1}{2}}{\bf z}}{=\joinrel=\joinrel=\joinrel=\joinrel=\joinrel=}\frac{{\bf y}^{\rm T}{\bf B}_1{\bf y}}{{\bf y}^{\rm T}{\bf B}_2{\bf y}}\leq 1,
	\end{align*}
	which means the maximal eigenvalue of ${\bf B}_2^{-\frac{1}{2}}{\bf B}_1{\bf B}_2^{-\frac{1}{2}}$ is no larger than 1. Thus, maximal eigenvalue of ${\bf B}_1^{\frac{1}{2}}{\bf B}_2^{-1}{\bf B}_1^{\frac{1}{2}}$ is no larger than 1.
	Hence, 
	\begin{align*}
	1\geq \frac{{\bf z}^{\rm T}{\bf B}_1^{\frac{1}{2}}{\bf B}_2^{-1}{\bf B}_1^{\frac{1}{2}}{\bf z}}{{\bf z}^{\rm T}{\bf z}}\stackrel{{\bf y}={\bf B}_1^{\frac{1}{2}}{\bf z}}{=\joinrel=\joinrel=\joinrel=\joinrel=\joinrel=}\frac{{\bf y}^{\rm T}{\bf B}_2^{-1}{\bf y}}{{\bf y}^{\rm T}{\bf B}_1^{-1}{\bf y}},
	\end{align*}
	which implies ${\bf B}_2^{-1}\preceq{\bf B}_1^{-1}$. The proof is complete.
\end{proof}

\begin{lemma}\textnormal{\textcolor{black}{(see \cite{lin2018separable})}}\label{timematspd}
	For any $\alpha\in(0,1)$, it holds that $l_0^{(\alpha)}>0$ and ${\bf T}+{\bf T}^{\rm T}\succ{\bf O}$.
\end{lemma}

The following proposition holds obviously.
\begin{proposition}\label{wghtsumbdlem}
	For positive numbers $\xi_i$, $\zeta_i$ $(1\leq i\leq m)$,
	it obviously holds that
	\begin{equation*}
	\min\limits_{1\leq i\leq m}\frac{\xi_i}{\zeta_i}\leq\bigg(\sum\limits_{i=1}^{m}\zeta_i\bigg)^{-1}\bigg(\sum\limits_{i=1}^{m}\xi_i\bigg)\leq\max\limits_{1\leq i\leq m}\frac{\xi_i}{\zeta_i}.
	\end{equation*}
\end{proposition}

For any invertible matrix ${\bf B}\in\mathbb{R}^{k\times k}$, define its condition number $\kappa_2({\bf B})$ by $$\kappa_2({\bf B}):=||{\bf B}^{-1}||_2||{\bf B}||_2.$$

\begin{theorem}\label{mainthm}
	Condition number of the preconditioned matrix ${\bf P}_l^{-1}{\bf A}{\bf P}_r^{-1}$ is uniformly bounded by a constant independent of $N$ and $J$, i.e.,
	\begin{equation*}
	\sup\limits_{N,J}\kappa_2({\bf P}_l^{-1}{\bf A}{\bf P}_r^{-1})\leq \sigma_0,
	\end{equation*}
	where $\sigma_0=\frac{\hat{a}}{\check{a}}$ is a positive constant independent of $N$ and $J$.
\end{theorem}

\begin{proof} 	
	Denote $\hat{\bf A}={\bf L}_a^{\frac{1}{2}}\otimes{\bf I}_N+{\bf L}_a^{-\frac{1}{2}}\otimes{\bf T}$. Recall that ${\bf P}_r=(\beta{\bf L}_{1})^{\frac{1}{2}}\otimes{\bf I}_N$. Then, it is clear that 
	\begin{align}
	&{\bf A}=\hat{\bf A}({\bf L}_a^{\frac{1}{2}}\otimes{\bf I}_N),\notag\\
	&({\bf P}_l^{-1}{\bf A}{\bf P}_r^{-1})({\bf P}_l^{-1}{\bf A}{\bf P}_r^{-1})^{\rm T}={\bf P}_l^{-1}\hat{\bf A}[({\bf L}_a^{\frac{1}{2}}(\beta{\bf L}_{1})^{-1}{\bf L}_a^{\frac{1}{2}})\otimes{\bf I}_N]\hat{\bf A}^{\rm T}{\bf P}_l^{-\rm T}.\notag
	\end{align}
	Moreover, it is easy to see that $({\bf L}_a^{\frac{1}{2}}{\bf L}_{1}^{-1}{\bf L}_a^{\frac{1}{2}})\otimes{\bf I}_N$ is similar to $({\bf L}_{1}^{-\frac{1}{2}}{\bf L}_a{\bf L}_{1}^{-\frac{1}{2}})\otimes{\bf I}_N$, which together with Lemma \ref{barlineqs}${\bf (ii)}$ implies that
	\begin{equation*}
	\sqrt{\frac{\check{a}}{\hat{a}}}{\bf I}_{NJ}=\frac{\check{a}}{\beta}{\bf I}_{NJ}\preceq({\bf L}_a^{\frac{1}{2}}(\beta{\bf L}_{1})^{-1}{\bf L}_a^{\frac{1}{2}})\otimes{\bf I}_N\preceq\frac{\hat{a}}{\beta}{\bf I}_{NJ}=\sqrt{\frac{\hat{a}}{\check{a}}}{\bf I}_{NJ}.
	\end{equation*}
    \textcolor{black}{ Here, ${\bf I}_{NJ}$ denotes the $(NJ)\times (NJ)$ identity matrix.}
	
	Hence,
	\begin{align}\label{part1esti}
	\sqrt{\frac{\check{a}}{\hat{a}}}{\bf P}_l^{-1}\hat{\bf A}\hat{\bf A}^{\rm T}{\bf P}_l^{-\rm T}\preceq({\bf P}_l^{-1}{\bf A}{\bf P}_r^{-1})({\bf P}_l^{-1}{\bf A}{\bf P}_r^{-1})^{\rm T}\preceq\sqrt{\frac{\hat{a}}{\check{a}}}{\bf P}_l^{-1}\hat{\bf A}\hat{\bf A}^{\rm T}{\bf P}_l^{-\rm T}.
	\end{align}
	It thus remains to estimate Rayleigh quotient of ${\bf P}_l^{-1}\hat{\bf A}\hat{\bf A}^{\rm T}{\bf P}_l^{-\rm T}$. Let ${\bf z}\in\mathbb{R}^{ J N\times 1}$ \textcolor{black}{denote} any non-zero vector. Then,
	\begin{align}
	\frac{{\bf z}^{\rm T}{\bf P}_l^{-1}\hat{\bf A}\hat{\bf A}^{\rm T}{\bf P}_l^{-\rm T}{\bf z}}{{\bf z}^{\rm T}{\bf z}}&\stackrel{{\bf y}={\bf P}_l^{-\rm T}{\bf z}}{=\joinrel=\joinrel=\joinrel=\joinrel=\joinrel=}\frac{{\bf y}^{\rm T}\hat{\bf A}\hat{\bf A}^{\rm T}{\bf y}}{{\bf y}^{\rm T}{\bf P}_l{\bf P}_l^{\rm T}{\bf y}}\notag\\
	&=\frac{{\bf y}^{\rm T}[{\bf L}_a\otimes{\bf I}_N+{\bf I}_{ J }\otimes({\bf T}+{\bf T}^{\rm T})+{\bf L}_a^{-1}\otimes({\bf T}{\bf T}^{\rm T})]{\bf y}}{{\bf y}^{\rm T}[(\beta{\bf L}_{1})\otimes{\bf I}_N+{\bf I}_{ J }\otimes({\bf T}+{\bf T}^{\rm T})+(\beta{\bf L}_{1})^{-1}\otimes({\bf T}{\bf T}^{\rm T})]{\bf y}},\label{part2preesti}
	\end{align}
	By Lemma \ref{timematspd}, we know that ${\bf T}+{\bf T}^{\rm T}\succ{\bf O}$. Since ${\bf T}$ is a lower triangular matrix with its diagonal entries all equal to $l_0^{(\alpha)}>0$, ${\bf T}$ is invertible and thus ${\bf T}{\bf T}^{\rm T}\succ{\bf O}$. That means the matrices appearing in the numerator and \textcolor{black}{the denominator} of right hand side of \eqref{part2preesti} are all positive definite. Thus, Proposition \ref{wghtsumbdlem} is applicable to estimating \eqref{part2preesti}.
	
	By Lemma \ref{barlineqs}${\bf (i)}$,
	\begin{equation}\label{firstterm}
	\sqrt{\frac{\check{a}}{\hat{a}}}=\frac{\check{a}}{\beta}\leq\frac{{\bf y}^{\rm T}({\bf L}_a\otimes{\bf I}_N){\bf y}}{{\bf y}^{\rm T}[(\beta{\bf L}_{1})\otimes{\bf I}_N]{\bf y}}\leq\frac{\hat{a}}{\beta}=\sqrt{\frac{\hat{a}}{\check{a}}}.
	\end{equation}
	By Lemma \ref{barlineqs}${\bf (i)}$ and Lemma \ref{spdneqs},
	\begin{align}
	\sqrt{\frac{\check{a}}{\hat{a}}}=\frac{\beta}{\hat{a}}=\frac{{\bf y}^{\rm T}[\hat{a}^{-1}{\bf L}_{1}^{-1}\otimes({\bf T}{\bf T}^{\rm T})]{\bf y}}{{\bf y}^{\rm T}[(\beta{\bf L}_{1})^{-1}\otimes({\bf T}{\bf T}^{\rm T})]{\bf y}}&\leq\frac{{\bf y}^{\rm T}[{\bf L}_a^{-1}\otimes({\bf T}{\bf T}^{\rm T})]{\bf y}}{{\bf y}^{\rm T}[(\beta{\bf L}_{1})^{-1}\otimes({\bf T}{\bf T}^{\rm T})]{\bf y}}\notag\\
	&\leq\frac{{\bf y}^{\rm T}[\check{a}^{-1}{\bf L}_{1}^{-1}\otimes({\bf T}{\bf T}^{\rm T})]{\bf y}}{{\bf y}^{\rm T}[(\beta{\bf L}_{1})^{-1}\otimes({\bf T}{\bf T}^{\rm T})]{\bf y}}=\frac{\beta}{\check{a}}=\sqrt{\frac{\hat{a}}{\check{a}}}.\label{thirdterm}
	\end{align}
	Applying Proposition \ref{wghtsumbdlem} to \eqref{part2preesti}, \eqref{firstterm} and \eqref{thirdterm}, we obtain that
	\begin{equation*}
	\sqrt{\frac{\check{a}}{\hat{a}}}=\min\left\{\frac{\check{a}}{\beta},1,\frac{\beta}{\hat{a}}\right\}\leq\frac{{\bf z}^{\rm T}{\bf P}_l^{-1}\hat{\bf A}\hat{\bf A}^{\rm T}{\bf P}_l^{-\rm T}{\bf z}}{{\bf z}^{\rm T}{\bf z}}\leq\max\left\{\frac{\hat{a}}{\beta},1,\frac{\beta}{\check{a}}\right\}=\sqrt{\frac{\hat{a}}{\check{a}}},
	\end{equation*}
	which together with \eqref{part1esti} implies that
	\begin{equation}\label{singvalsesti}
	\frac{\check{a}}{\hat{a}}{\bf I}_{N J }\preceq({\bf P}_l^{-1}{\bf A}{\bf P}_r^{-1})({\bf P}_l^{-1}{\bf A}{\bf P}_r^{-1})^{\rm T}\preceq	\frac{\hat{a}}{\check{a}}{\bf I}_{N J }.
	\end{equation}
	\eqref{singvalsesti} implies that
	\begin{align*}
	\kappa_2({\bf P}_l^{-1}{\bf A}{\bf P}_r^{-1})&\leq\sqrt{	\left(\frac{\hat{a}}{\check{a}}\right)\Bigg/\left(\frac{\check{a}}{\hat{a}}\right)}=\frac{\hat{a}}{\check{a}}=\sigma_0.
	\end{align*}
	The proof is complete.
\end{proof}

With the condition number estimation in Theorem \ref{mainthm}, one can immediately get the following corollary.
\begin{corollary}\textnormal{(see \cite[Theorem 38.5]{trefethen1997numerical})}\label{maincoroll}
	The normalized conjugate gradient (NCG) solver for the preconditioned system \eqref{twosidedpreclinsys} has a linear convergence rate $\frac{\sigma_0-1}{\sigma_0+1}<1$ with $\sigma_0$ given in Theorem \ref{mainthm}, which is independent of $N$ and $J$.
\end{corollary}

\section{The Implementation}\label{implementsection}
In this section, we propose a fast implementation of Krylov subspace solver for solving the preconditioned system \eqref{twosidedpreclinsys}. To fast implement a Krylov subspace solver, it suffices to fast implement the underlying matrix-vector product. In other words, we will discuss in this section how to fast compute a matrix-vector product ${\bf P}_l^{-1}{\bf A}{\bf P}_r^{-1}{\bf v}$ for an arbitrarily given vector ${\bf v}$. Also, the fast computation of \eqref{scalling} with given $\hat{\bf u}$ will be discussed in \textcolor{black}{this} section.

We firstly introduce some preliminaries.
	
\textcolor{black}{Recall that the physical domain $\Omega=\prod\limits_{i=1}^{d}(\check{c}_i,\hat{c}_i)$ and the spatial operator is discretized by \textcolor{black}{the} central difference scheme on uniform grid.} Partition the interval $[\check{c}_i,\hat{c}_i]$ into $m_i$ uniform sub-intervals \textcolor{black}{with} $h_i=(\hat{c_i}-\check{c}_i)/(m_i+1)$ as the stepsize. Clearly, $J=\prod\limits_{i=1}^{d}m_i$. Denote
\begin{align*}
&m_1^{-}=m_d^{+}=1,\quad m_i^{-}=\prod\limits_{j=1}^{i-1}m_j,\quad m_i^{+}=\prod\limits_{j=i+1}^{d}m_j,\quad  i=2,3,...,d-1,\\
&{\bf W}_m=\left[
\begin{array}
	[c]{ccccc}
	2&-1 &   &  & \\
	-1& 2&-1  &&\\
	~&\ddots &\ddots&\ddots&~\\
	~ & ~&-1&2 &-1\\
	~& ~&~& -1 & 2
\end{array}
\right]\in\mathbb{R}^{m\times m}.
\end{align*}

Then, it is well-known that
\begin{align*}
{\bf L}_1=\sum\limits_{i=1}^{d}{\bf I}_{m_i^{-}}\otimes(h_i^{-2}{\bf W}_{m_i})\otimes{\bf I}_{m_i^{+}}.
\end{align*}

Denote
\begin{equation}\label{1dsinemat}
	{\bf S}_m:=\sqrt{\frac{2}{m+1}}\left[\sin\left(\frac{ij\pi}{m+1}\right)\right]_{i,j=1}^{m}.
\end{equation}
${\bf S}_m$ is called $m\times m$ one-dimension sine transform matrix. 

Some properties of ${\bf S}_m$ is given in the following proposition.
\begin{proposition}\label{smprop}
\begin{description}
\item[(i)] For any $m$, ${\bf S}_m$ is real orthogonal and symmetric, i.e., ${\bf S}_m{\bf S}_m^{\rm T}={\bf I}_m$ and ${\bf S}_m={\bf S}_{m}^{\rm T}$.
\item[(ii)]For any $m$, ${\bf W}_m$ is diagonalizable by ${\bf S}_m$, i.e.,
\begin{equation*}
{\bf W}_m={\bf S}_m{\bf D}_m{\bf S}_m^{\rm T},\quad {\bf D}_m={\rm diag}\left(4\sin^2\left(\frac{i\pi}{2(m+1)}\right)\right)_{i=1}^{m}.
\end{equation*}
\end{description}
\end{proposition}
\begin{proof}
It can be proven by straight forward calculation.
\end{proof}

Define $d$-dimension sine transform matrix as
\begin{equation*}
	{\bf Q}:={\bf S}_{m_1}\otimes{\bf S}_{m_2}\otimes\cdots\otimes{\bf S}_{m_d}.
\end{equation*}

Proposition \ref{smprop}${\bf (i)}$ immediately implies Proposition \ref{fstmatprop}.
\begin{proposition}\label{fstmatprop}
	${\bf Q}$ is real orthogonal and symmetric, i.e., ${\bf Q}^{\rm T}{\bf Q}={\bf I}_{J}$ and ${\bf Q}={\bf Q}^{\rm T}$.
\end{proposition}

Moreover, Proposition \ref{smprop} also implies that
\begin{equation}\label{l1diagform}
{\bf L}_1={\bf Q}{\bf \Lambda}{\bf Q}^{\rm T}={\bf Q}{\bf \Lambda}{\bf Q},
\end{equation}
with
\begin{equation*}
{\bf \Lambda}=\sum\limits_{i=1}^{d}{\bf I}_{m_i^{-}}\otimes(h_i^{-2}{\bf D}_{m_i})\otimes{\bf I}_{m_i^{+}}.
\end{equation*}
Clearly, ${\bf \Lambda}$ is a positive definite diagonal matrix with its diagonal entries explicitly known.

By Proposition \ref{fstmatprop} and \eqref{l1diagform}, we know that
\begin{equation*}
{\bf P}_r^{-1}=({\bf Q}\otimes{\bf I}_N)(\beta^{-\frac{1}{2}}{\bf \Lambda}^{-\frac{1}{2}}\otimes {\bf I}_N)({\bf Q}\otimes{\bf I}_N).
\end{equation*} 
%
\begin{lemma}\textnormal{(see \cite[Algorithm 1.4.2]{goluvan2013})}\label{1dsineimplemt}
	For any positive integer $m$ and an arbitrarily given ${\bf y}\in\mathbb{R}^{m\times 1}$, the computation of the matrix-vector product ${\bf S}_m{\bf y}$ requires $\mathcal{O}(m\log m)$ flops. 
\end{lemma}
\begin{proof}
	Actually, \cite[Algorithm 1.4.2]{goluvan2013} computes ${\bf S}_m{\bf y}$ by acting a fast Fourier transform (FFT) on an extended vector $[0;{\bf y};0;-{\bf y}(m);-{\bf y}(m-1);\cdots;-{\bf y}(1)]\in\mathbb{R}^{2(m+1)\times 1}$. Meanwhile, FFT of a length-$n$  vector requires $\mathcal{O}(n\log n)$ flops; see, e.g., the chirp $z$-transform \cite{rabiner1969chirp}. This is how we obtain Lemma \ref{1dsineimplemt}.
\end{proof}

\textcolor{black}{By properties of Kronecker product and Lemma \ref{1dsineimplemt}, we know that the computation of  ${\bf P}_r^{-1}{\bf v}$  for a given vector  ${\bf v}\in\mathbb{R}^{NJ\times 1}$ requires $\mathcal{O}(NJ\log J)$ flops. That means the computation of \eqref{scalling} requires $\mathcal{O}(NJ\log J)$ flops once $\hat{\bf u}$ is given.}

In what follows, we discuss the fast computation of a matrix-vector product $\hat{\bf v}={\bf P}_l^{-1}{\bf A}{\bf P}_r^{-1}{\bf v}$ for a given vector ${\bf v}\in\mathbb{R}^{NJ\times 1}$. Clearly, the computation of $\hat{\bf v}$ can be divided into the following three steps.
\begin{align}
&{\rm Step~1}: \quad  {\rm Compute~}\dot{\bf v}={\bf P}_r^{-1}{\bf v},\label{step1}\\
&{\rm Step~2}: \quad  {\rm Compute~} \ddot{\bf v}={\bf A}\dot{\bf v},\label{step2}\\
&{\rm Step~3}: \quad  {\rm Compute~}\hat{\bf v}={\bf P}_l^{-1}\ddot{\bf v}.\label{step3}
\end{align}
\textcolor{black}{Similar to the above discussion for computation of \eqref{scalling}, we know that the computation of \eqref{step1} requires $\mathcal{O}(NJ\log J)$ flops. It remains to discuss the computation of \eqref{step2} and \eqref{step3}.}

\textcolor{black}{Notice that $\ddot{\bf v}={\bf A}\dot{\bf v}=({\bf L}_a\otimes{\bf I}_N)\dot{\bf v}+({\bf I}_J\otimes{\bf T})\dot{\bf v}$. It is clear that ${\bf L}_a$ is a sparse matrix with $\mathcal{O}(J)$ nonzero entries. Moreover, ${\bf T}$ is a Toeplitz matrix  whose matrix-vector product requires $\mathcal{O}(N\log N)$ flops; see, e.g., \cite{mng2004}. By properties of the Kronecker product, we know that the computation of \eqref{step2} requires $\mathcal{O}(JN\log N)$ flops.}

It remains to discuss the fast computation of \eqref{step3}. Rewrite ${\bf \Lambda}$ in \eqref{l1diagform} as
\begin{equation*}
{\bf \Lambda}=\diag(\lambda_i)_{i=1}^{J}.
\end{equation*}
Clearly, $\lambda_i$'s ($i=1,2,...,J$) are all positive numbers by Proposition \ref{smprop}${\bf (ii)}$. By \eqref{l1diagform} and Lemma \ref{fstmatprop}, we know that
\begin{equation}\label{step3decomp}
{\bf P}_l^{-1}\ddot{\bf v}=({\bf Q}\otimes{\bf I}_N)\textcolor{black}{\blockdiag}({\bf T}_i^{-1})_{i=1}^{ J }({\bf Q}\otimes{\bf I}_N)\ddot{\bf v},
\end{equation}
where
\begin{equation*}
{\bf T}_i=(\beta{\lambda}_i)^{-\frac{1}{2}}{\bf T}+(\beta{\lambda}_i)^{\frac{1}{2}}{\bf I}_N,\quad i= 1,2,...,J.
\end{equation*}
Then, the computation of \eqref{step3} is equivalent to the following three sub-steps
\begin{align}
&\ddot{\bf v}^{\prime}=({\bf Q}\otimes{\bf I}_N)\ddot{\bf v},\label{step31}\\
&\ddot{\bf v}^{\prime\prime}=\textcolor{black}{\blockdiag}({\bf T}_i^{-1})_{i=1}^{ J }\ddot{\bf v}^{\prime},\label{step32}\\
&\hat{\bf v}=({\bf Q}\otimes{\bf I}_N)\ddot{\bf v}^{\prime\prime}.\label{step33}
\end{align}
\textcolor{black}{By properties of the Kronecker product and Lemma \ref{1dsineimplemt}, the computation of \eqref{step31} and \eqref{step33} requires $\mathcal{O}(NJ\log J)$ flops. It remains to discuss the fast computation of \eqref{step32}. The matrix in \eqref{step32} is a block diagonal matrix with inverse of ${\bf T}_i$ ($1\leq i\leq J$) as diagonal blocks. Notice that these ${\bf T}_i$'s are all invertible lower triangular Toeplitz (ILTT) matrices. Hence, for fast computation of \eqref{step32}, it suffices to show that for any $n\times n$ ILTT matrix ${\bf G}$ and an arbitrarily given vector ${\bf y}\in\mathbb{R}^{n\times 1}$, the matrix-vector product ${\bf G}^{-1}{\bf y}$ can be fast computed.
\begin{lemma}\textnormal{(see \cite{commenges1984fast})}\label{invlttlm}
	For any $n\times n $ ILTT matrix ${\bf G}$, its inverse ${\bf G}^{-1}$ is also an ILTT matrix and ${\bf G}^{-1}(:,1)$ can be computed from ${\bf G}(:,1)$ within $\mathcal{O}(n\log n)$ flops.
\end{lemma}}

\textcolor{black}{Actually, ${\bf T}_i^{-1}(:,1)$ ($i=1,2,...,J$) can be computed and stored before solving the two-sided preconditioned system. From Lemma \ref{invlttlm}, we know that the computation of ${\bf T}_i^{-1}(:,1)$'s ($i=1,2,...,J$) and \eqref{step32} requires $\mathcal{O}(JN\log N)$ flops. Summing over the operation cost for \eqref{step31}--\eqref{step33}, we see that the computation of \eqref{step3} requires $\mathcal{O}(NJ\log (NJ))$ flops.}

\textcolor{black}{Summing up the above discussion, for a given vector  ${\bf v}\in\mathbb{R}^{NJ\times 1}$, the matrix-vector product ${\bf P}_l^{-1}{\bf A}{\bf P}_r^{-1}{\bf v}$ can be fast computed within  $\mathcal{O}(NJ\log(NJ))$ flops.}

As supported by Corollary \ref{maincoroll} and the numerical results in \textcolor{black}{Section \ref{experimentsection}}, we see that two Krylov subspace solvers, NCG and restarted \textcolor{black}{generalized minimal residual method (GMRES)} converge within an iteration number independent of $N$ and $J$, which means the Krylov subspace solvers for the preconditioned system \eqref{twosidedpreclinsys} requires operation cost proportional to that of one preconditioned matrix-vector product. \textcolor{black}{Hence, the Krylov subspace solvers for solving the preconditioned system \eqref{twosidedpreclinsys} requires only $\mathcal{O}(NJ\log(NJ))$ flops, which is nearly optimal as $NJ$ is the number of unknown.}
%

\begin{remark}\label{fastdirctsolrem}
It is clear that ${\bf A}={\bf P}_r{\bf P}_l$ when $a$ is a constant. In such case, the unpreconditioned all-at-once system \eqref{allatoncesysst} can be directly solved by ${\bf u}={\bf P}_l^{-1}({\bf P}_r^{-1}{\bf f})$. Clearly, the fast computation of ${\bf P}_l^{-1}({\bf P}_r^{-1}{\bf f})$ are already discussed above (see the implementation details for \eqref{step1} and \eqref{step3}), which requires $\mathcal{O}(NJ\log(NJ))$ flops. This is a fast direct PinT solver for the all-at-once system \eqref{allatoncesysst} in the case of $a({\bf x})$ being a constant.
\end{remark}

\section{Numerical Experiments}\label{experimentsection}
In this section, we test the proposed solver on several examples and compare it with the state-of-art solvers to show its efficiency. 
All numerical experiments are performed via MATLAB R2018a on a workstation equipped with dual Xeon Gold 6146 12-Cores 3.2GHz CPUs, NVIDIA Quadro P2000 GPU, 384GB RAM running CentOS Linux version 7.

Define the error as
\begin{equation*}
{\rm E}_{N,J}:=||{\bf u}^{*}-{\bf u}_{{\rm exact}}||_{\infty},
\end{equation*}
where ${\bf u}^{*}$ denotes some numerical solution; ${\bf u}_{{\rm exact}}$ denotes the exact solution of the non-local evolutionary equation on the whole time-space grid. Denote by CPU, the computational time in unit of second. Denote by DoF, degree of freedom (i.e., the number of unknowns $NJ$).

\begin{example}\label{constexmpl}
	{\rm 	
		In this example, we consider the non-local evolutionary equation \eqref{sub-diffusioneq}--\eqref{initial condition} with
		\begin{align*}
		&\Omega=(0,\pi)\times(0,\pi),~ T=1,~{\bf x}=(x,y),~ u(x,y,t)=\sin(x)\sin(y)t^2+x(\pi-x)y(\pi-y),\\
		& a(x,y)\equiv 1,\quad f(x,y,t)=\sin(x)\sin(y)\left[\frac{2t^{2-\alpha}}{\Gamma(3-\alpha)}+2t^2\right]+2[x(\pi-x)+y(\pi-y)].
		\end{align*}
		Notice that Example \ref{constexmpl} has a constant coefficient $a({\bf x})\equiv 1$. As introduced in Section \ref{introduction}, \textcolor{black}{a fast kernel compression based  time-stepping method}  was proposed in \cite{jiang2017fast} for the non-local evolutionary equation. We denote  this fast time-stepping method by FKC. When applying FKC to Example \ref{constexmpl}, the resulting time-stepping linear systems can be solved by \textcolor{black}{the} fast Poisson solver as $a$ is a constant. We denote the FKC with fast Poisson solver by FKC-FPS.
	    Another efficient solver introduced in Section \ref{introduction} is the fast approximation method proposed in \cite{lin2016fast}, in which the approximated all-at-once system is fast block diagonalizable with complex diagonal blocks.  When $a$ is a constant, the complex diagonal system can be solved by \textcolor{black}{the} fast Poisson solver. We denote the fast approximation method proposed in \cite{lin2016fast} with \textcolor{black}{the} fast Poisson solver by FAM-FPS. Moreover, as discussed in Remark \ref{fastdirctsolrem}, in the case of $a$ being a constant, our proposed implementation in Section \ref{implementsection} is exactly a fast direct solver for the all-at-once system. We denote our proposed fast direct solver for all-at-once system by FDS-AAO. We test FDS-AAO, FAM-FPS and FKC-FPS on Example \ref{constexmpl}, the results of which is listed in Table \ref{csntexpltb}. Table \ref{csntexpltb} shows that (i) the two PinT solvers, FDS-AAO and FAM-FPS, are  more efficient than the time-stepping solver FKC-FPS in terms of computational time; (ii) FDS-AAO and FKC-FPS are generally more accurate than FAM-FPS because of the additional matrix approximation error introduced in FAM-FPS. Overall, the proposed FDS-AAO solver performs the best among the three solvers. \textcolor{black}{From Table \ref{csntexpltb}, we also see that FDS-AAO requires less computational time than FAM-FPS, although they are both PinT solvers. This is because that  FAM-FPS involves complex arithmetic (i.e., the operations on complex numbers) while FDS-AAO  only involves real arithmetic.}
		\begin{table}[H]
			\begin{center}
				\caption{Performance of FDS-AAO, FAM-FPS and FKC-FPS with $J=2^{14}$.}\label{csntexpltb}
				\setlength{\tabcolsep}{0.73em}
				\begin{tabular}[c]{ccc|cc|cc|cc}
					\hline
					\multirow{2}{*}{$\alpha$} &\multirow{2}{*}{$N+1$} &\multirow{2}{*}{DoF}&\multicolumn{2}{c|}{FDS-AAO} & \multicolumn{2}{c|}{FAM-FPS}  &\multicolumn{2}{c}{FKC-FPS}\\
					\cline{4-9}
					&&&$\mathrm{CPU}$&${\rm E}_{N,J}$&$\mathrm{CPU}$&${\rm E}_{N,J}$&$\mathrm{CPU}$& ${\rm E}_{N,J}$\\
					\hline
					\multirow{4}{*}{0.1}
					&$2^{14}$ &268419072 &34.72s &3.19e-5 &45.84s           &3.24e-5&86.39s    &3.19e-5\\
					&$2^{15}$ &536854528 &68.55s &3.19e-5 &95.37s           &3.25e-5&219.67s   &3.19e-5\\
					&$2^{16}$ &1073725440&149.32s&3.19e-5 &187.69s          &3.22e-5&354.20s   &3.19e-5\\
					&$2^{17}$ &2147467264&287.65s&3.19e-5 &402.45s          &3.22e-5&734.20s   &3.19e-5\\
					\hline
					\multirow{4}{*}{0.5}
					&$2^{14}$ &268419072 &34.24s &2.77e-5 &49.26s           &2.55e-5&86.04s    &2.77e-5\\
					&$2^{15}$ &536854528 &68.41s &2.76e-5 &101.86s          &2.98e-5&178.63s   &2.76e-5\\
					&$2^{16}$ &1073725440&139.46s&2.76e-5 &198.94s          &3.00e-5&373.86s   &2.76e-5\\
					&$2^{17}$ &2147467264&298.40s&2.76e-5 &403.99s          &2.75e-5&748.64s   &2.76e-5\\
					\hline
					\multirow{4}{*}{0.9}
					&$2^{14}$ &268419072 &34.25s &3.12e-5 &48.49s           &7.27e-5&87.52s    &3.12e-5\\
					&$2^{15}$ &536854528 &68.65s &2.67e-5 &92.86s           &7.51e-5&174.42s   &2.67e-5\\
					&$2^{16}$ &1073725440&139.84s&2.46e-5 &185.44s          &2.59e-5&376.14s   &2.46e-5\\
					&$2^{17}$ &2147467264&301.20s&2.36e-5 &407.93s          &9.35e-5&735.85s   &2.36e-5\\
					\hline
				\end{tabular}
			\end{center}
		\end{table}
		
	}
\end{example}

The rest of this section is devoted to testing efficiency of the proposed two-sided preconditioning technique on examples with non-constant $a({\bf x})$. Since the three solvers, FDS-AAO, FAM-FPS and FKC-FPS tested in Example \ref{csntexpltb} are only available for the non-local evolutionary equation with constant coefficient, they will not be tested in the remaining content.
%

Denote by `Iter', the iteration number of an iterative solver. For all Krylov subspace solvers tested in this section, we set zero vector as initial guess and set $||{\bf r}^k||_2\leq 10^{-7}||{\bf r}^0||_2$ as stopping criterion if not specified, where ${\bf r}^k$ denotes the residual vector at $k$-th iteration.

As described in Section \ref{introduction}, in \cite{lin2016fast}, the fast approximation method with \textcolor{black}{a} multigrid spatial solver is proposed for solving the non-local evolutionary equation with non-constant $a({\bf x})$. We denote the fast approximation method with \textcolor{black}{the} multigrid spatial solver by FAM-MG.
As indicated by Corollary \ref{mainthm}, NCG solver can be employed to solve the two-sided preconditioned system \eqref{twosidedpreclinsys}. We denote the NCG solver for the two-sided preconditioned system by NCG-2S. Besides, GMRES solver can be also used to solve the two-sided preconditioned system \eqref{twosidedpreclinsys}, as it does not require symmetry of the linear system. We denote GMRES solver for the two-sided preconditioned system \eqref{twosidedpreclinsys} by GMRES-2S. The GMRES solver employed in this paper is a restarted version with restarting number 50.

\begin{example}\label{varcoeff2dexmpl}
	{\rm
		Consider the problem \eqref{sub-diffusioneq}--\eqref{initial condition} with
		\begin{align*}
		&\Omega=(0,1)^2,\quad T=1,\quad{\bf x}=(x,y),\quad a(x,y)=40+x^{3.5}+y^{3.5},\\
		& f(x,y,t)=\sin(\pi x)\sin(\pi y)\left[\frac{2t^{2-\alpha}}{\Gamma(3-\alpha)}+2\pi^2at^2\right]\\
		&\qquad\qquad\quad-\pi t^2\left[(\partial_xa)\cos(\pi x)\sin(\pi y)+(\partial_ya)\sin(\pi x)\cos(\pi y)\right],
		\end{align*}
		the analytical solution of which is $u(x,y,t)=\sin(\pi x)\sin(\pi y)t^2$.
		We test GMRES-2S, FAM-MG and NCG-2S on Example \ref{varcoeff2dexmpl}, the results of which are listed in Tables \ref{varcoeff2dexmpltb1}--\ref{varcoeff2dexmpltb2}. Tables \ref{varcoeff2dexmpltb1}--\ref{varcoeff2dexmpltb2} \textcolor{black}{show} that (i) NCG-2S has a bounded iteration number, which illustrates a matrix-size-independent convergence rate and supports Corollary \ref{maincoroll}; (ii) GMRES-2S and NCG-2S are more efficient than FAM-MG in terms of CPU while accuracy of the three solvers are almost the same, which demonstrates the efficiency of the proposed two-sided preconditioning technique. Additionally, we note that GMRES-2S and NCG-2S \textcolor{black}{converge} equally fast (i.e., the iteration number is the same) while CPU of NCG-2S is roughly twice as much as that of GMRES-2S. This is due to the additional matrix transpose involved in NCG-2S, which doubles the computational cost of each matrix-vector product.
		\begin{table}[H]
			\begin{center}
				\caption{Results of different solvers for solving Example \ref{varcoeff2dexmpl} when  $J=16129$.}\label{varcoeff2dexmpltb1}
				\setlength{\tabcolsep}{0.48em}
				\begin{tabular}[c]{ccc|ccc|ccc|cc}
					\hline
					\multirow{2}{*}{$\alpha$} &\multirow{2}{*}{$N+1$} &\multirow{2}{*}{DoF}& \multicolumn{3}{c|}{GMRES-2S}   &\multicolumn{3}{c|}{NCG-2S}& \multicolumn{2}{c}{FAM-MG}\\
					\cline{4-11}
					&&&$\mathrm{Iter}$&$\mathrm{CPU}$&$\mathrm{E}_{N,J}$&$\mathrm{Iter}$&$\mathrm{CPU}$&$\mathrm{E}_{N,J}$&$\mathrm{CPU}$&$\mathrm{E}_{N,J}$\\
					\hline
					\multirow{4}{*}{0.1} 
					&$2^{10}$ &16499967 &4   &19.45s  &4.97e-5 &4   &37.49s   &4.97e-5&38.37s &4.97e-5 \\
					&$2^{11}$ &33016063 &4   &38.41s  &4.97e-5 &4   &73.59s   &4.97e-5&75.90s &4.97e-5 \\
					&$2^{12}$ &66048255 &4   &76.99s  &4.97e-5 &4   &146.32s  &4.97e-5&153.04s&4.97e-5 \\
					&$2^{13}$ &132112639&4   &154.75s &4.97e-5 &4   &298.64s  &4.97e-5&304.61s&4.97e-5 \\
					\hline
					\multirow{4}{*}{0.5} 
					&$2^{10}$ &16499967 &4   &19.44s  &4.96e-5 &4   &37.26s   &4.97e-5&37.92s &4.96e-5 \\
					&$2^{11}$ &33016063 &4   &38.34s  &4.96e-5 &4   &73.54s   &4.97e-5&76.21s &4.96e-5 \\
					&$2^{12}$ &66048255 &4   &77.31s  &4.96e-5 &4   &147.07s  &4.97e-5&152.75s&4.96e-5 \\
					&$2^{13}$ &132112639&4   &154.97s &4.96e-5 &4   &297.90s  &4.97e-5&308.87s&4.96e-5 \\
					\hline
					\multirow{4}{*}{0.9}
					&$2^{10}$ &16499967 &4   &19.53s  &5.01e-5 &4   &37.52s   &5.01e-5&37.91s &5.01e-5 \\
					&$2^{11}$ &33016063 &4   &38.49s  &4.99e-5 &4   &73.41s   &4.98e-5&75.96s &4.99e-5 \\
					&$2^{12}$ &66048255 &4   &77.78s  &4.97e-5 &4   &148.31s  &4.97e-5&164.23s&4.97e-5 \\
					&$2^{13}$ &132112639&4   &155.37s &4.97e-5 &4   &297.88s  &4.96e-5&351.57s&4.97e-5 \\
					\hline
				\end{tabular}
			\end{center}
		\end{table}
		
		\begin{table}[H]
			\begin{center}
				\caption{Results of different solvers for solving Example \ref{varcoeff2dexmpl} when  $N+1=2^6$.}\label{varcoeff2dexmpltb2}
				\setlength{\tabcolsep}{0.43em}
				\begin{tabular}[c]{ccc|ccc|ccc|cc}
					\hline
					\multirow{2}{*}{$\alpha$} &\multirow{2}{*}{$J$}  &\multirow{2}{*}{DoF}& \multicolumn{3}{c|}{GMRES-2S}   &\multicolumn{3}{c|}{NCG-2S}& \multicolumn{2}{c}{FAM-MG}\\
					\cline{4-11}
					&&&$\mathrm{Iter}$&$\mathrm{CPU}$&$\mathrm{E}_{N,J}$&$\mathrm{Iter}$&$\mathrm{CPU}$&$\mathrm{E}_{N,J}$&$\mathrm{CPU}$&$\mathrm{E}_{N,J}$\\
					\hline
					\multirow{4}{*}{0.1} 
					&65025   &4096575  &4   &5.35s   &1.25e-5 &4   &9.81s    &1.25e-5&10.30s &1.25e-5 \\
					&261121  &16450623 &4   &19.70s  &3.16e-6 &4   &36.33s   &3.18e-6&42.38s &3.15e-6 \\
					&1046529 &65931327 &4   &74.26s  &8.29e-7 &4   &137.64s  &8.54e-7&184.96s&8.30e-7 \\
					&4190209 &263983167&4   &295.96s &2.47e-7 &4   &544.08s  &2.72e-7&772.89s&2.55e-7 \\
					\hline
					\multirow{4}{*}{0.5} 
					&65025   &4096575  &4   &5.38s   &1.36e-5 &4   &9.62s    &1.36e-5&10.47s &1.36e-5 \\
					&261121  &16450623 &4   &19.62s  &4.25e-6 &4   &36.43s   &4.27e-6&42.98s &4.25e-6 \\
					&1046529 &65931327 &4   &73.35s  &1.92e-6 &4   &137.48s  &1.95e-6&186.68s&1.92e-6 \\
					&4190209 &263983167&4   &298.47s &1.34e-6 &4   &545.14s  &1.37e-6&780.49s&1.36e-6 \\
					\hline
					\multirow{4}{*}{0.9}
					&65025   &4096575  &4   &5.23s   &2.39e-5 &4   &9.70s    &2.40e-5&9.91s  &2.47e-5 \\
					&261121  &16450623 &4   &19.59s  &1.46e-5 &4   &36.42s   &1.47e-5&44.18s &1.46e-5 \\
					&1046529 &65931327 &4   &73.41s  &1.23e-5 &4   &137.63s  &1.23e-5&190.48s&1.23e-5 \\
					&4190209 &263983167&4   &294.64s &1.17e-5 &4   &548.21s  &1.18e-5&768.26s&1.17e-5 \\
					\hline
				\end{tabular}
			\end{center}
		\end{table}
	}
\end{example}

Actually, not only in Example \ref{varcoeff2dexmpl}, GMRES-2S always converges no slower than NCG-2S in later examples. Hence, in the later examples, the results of NCG-2S are not listed. To demonstrate that the proposed two-sided preconditioning \textcolor{black}{technique} significantly improves the convergence rate of \textcolor{black}{the} Krylov subspace solver, we also test the unpreconditioned GMRES method in Example \ref{discticoeff3dexmpl}. We denote the unpreconditioned GMRES method by GMRES-${\bf I}$. 
In FAM-MG, the multigrid method is originally proposed for two-dimensional spatial \textcolor{black}{problems}. However, the multigrid method can also be extended to solving three-dimensional spatial \textcolor{black}{problems} by using the damped Jacobi smoother with $0.5$ as damping factor. With this extended multigrid spatial solver, FAM-MG can be used to solve  Example \ref{discticoeff3dexmpl}, a three-spatial-dimension non-local evolutionary equation.
As the exact solution of Example \ref{discticoeff3dexmpl} is unknown, we use the following residual quantity to measure the accuracy of different solvers:
\begin{equation*}
{\rm RES}=\frac{||{\bf f}-{\bf A}{\bf u}^{*}||_2}{||{\bf f}||_2},
\end{equation*}
where ${\bf u}^{*}$ denotes some approximate solution to the original all-at-once linear system \eqref{allatoncesysst}. 
\begin{example}\label{discticoeff3dexmpl}
	{\rm
		Consider the problem \eqref{sub-diffusioneq}--\eqref{initial condition} with
		\begin{align*}
		& \Omega=(0,1)^3,\quad T=1,\quad {\bf x}=(x,y,z),\quad u(\cdot,0)|_{\Omega}\equiv 0,\quad u|_{\partial\Omega\times(0,T]}\equiv 0\\
		&a(x,y,z)=\begin{cases}
		2,\quad x<0.5,\\
		2.5,\quad x\geq 0.5,
		\end{cases}\\
		& f(x,y,z,t)=xyz(1-x)(1-y)(1-z)\left[t^2+\frac{2t^{2-\alpha}}{\Gamma(3-\alpha)}\right].
		\end{align*}
		We test GMRES-2S, GMRES-${\bf I}$ and FAM-MG on Example \ref{discticoeff3dexmpl}, the results of which are listed in Tables \ref{expl3ddisctctb1}--\ref{expl3ddisctctb2}. Tables \ref{expl3ddisctctb1}--\ref{expl3ddisctctb2} shows that (i) \textcolor{black}{the} convergence rate of GMRES-2S is independent of matrix size for Example \ref{discticoeff3dexmpl}; (ii) GMRES-2S is the most efficient one among the three solvers in terms of CPU cost. Moreover, GMRES-2S converges much faster than GMRES-${\bf I}$ and the iteration number of GMRES-2S is more stable than that of GMRES-${\bf I}$ (especially in Table \ref{expl3ddisctctb2}), which demonstrates that the proposed two-sided preconditioning \textcolor{black}{technique} significantly improves  convergence rate and robustness of \textcolor{black}{the} Krylov subspace solvers for the discrete TFSDE problem. Additionally, we note that the coefficient $a$ in Example \ref{discticoeff3dexmpl} has a jump. That means the performance of the proposed two-sided preconditioning technique does not rely on smoothness of $a(\cdot)$.
		\begin{table}[H]
			\begin{center}
				\caption{Results of different solvers for solving Example \ref{discticoeff3dexmpl} when $J=2048383$.}\label{expl3ddisctctb1}
				\setlength{\tabcolsep}{0.48em}
				\begin{tabular}[c]{ccc|ccc|ccc|cc}
					\hline
					\multirow{2}{*}{$\alpha$} &\multirow{2}{*}{$N+1$} &\multirow{2}{*}{DoF}& \multicolumn{3}{c|}{GMRES-2S} &\multicolumn{3}{c|}{GMRES-${\bf I}$}& \multicolumn{2}{c}{FAM-MG}  \\
					\cline{4-11}
					&&&$\mathrm{Iter}$&$\mathrm{CPU}$&${\rm RES}$&$\mathrm{Iter}$&$\mathrm{CPU}$&${\rm RES}$&$\mathrm{CPU}$&${\rm RES}$\\
					\hline
					\multirow{4}{*}{0.1} 
					&$2^1$ &4096766  &5   &8.44s   &5.71e-8 &1069&798.66s  &9.91e-8&42.40s &7.58e-8 \\
					&$2^2$ &8193532  &5   &14.10s  &5.72e-8 &1069&1335.83s &9.88e-8&104.67s&7.55e-8 \\
					&$2^3$ &16387064 &5   &25.20s  &5.73e-8 &1068&2329.51s &9.98e-8&191.45s&2.08e-7 \\
					&$2^4$ &32774128 &5   &48.58s  &5.73e-8 &1068&4413.97s &9.98e-8&369.92s&2.04e-7 \\
					\hline
					\multirow{4}{*}{0.5} 
					&$2^1$ &4096766  &5   &8.16s   &5.82e-8 &1065&796.62s  &9.89e-8&43.17s &8.12e-8 \\
					&$2^2$ &8193532  &5   &13.98s  &5.89e-8 &1062&1329.30s &9.98e-8&96.41s &2.07e-7 \\
					&$2^3$ &16387064 &5   &25.30s  &5.94e-8 &1061&2365.20s &9.96e-8&180.88s&5.68e-7 \\
					&$2^4$ &32774128 &5   &48.58s  &5.96e-8 &1061&4442.44s &9.90e-8&341.37s&3.04e-7 \\
					\hline
					\multirow{4}{*}{0.9}
					&$2^1$ &4096766  &5   &8.23s   &5.96e-8 &1063&779.27s  &9.88e-8&43.28s &1.06e-7 \\
					&$2^2$ &8193532  &5   &14.16s  &6.23e-8 &1059&1320.54s &9.90e-8&93.66s &2.33e-7 \\
					&$2^3$ &16387064 &5   &25.36s  &6.47e-8 &1060&2344.33s &9.90e-8&180.13s&1.94e-6 \\
					&$2^4$ &32774128 &5   &48.62s  &6.63e-8 &1062&4435.52s &9.97e-8&333.68s&1.87e-6 \\
					\hline
				\end{tabular}
			\end{center}
		\end{table}
		
		\begin{table}[H]
			\begin{center}
				\caption{Results of different solvers for solving Example \ref{discticoeff3dexmpl} when  $N=2^6$.}\label{expl3ddisctctb2}
				\setlength{\tabcolsep}{0.35em}
				\begin{tabular}[c]{ccc|ccc|ccc|cc}
					\hline
					\multirow{2}{*}{$\alpha$} &\multirow{2}{*}{$J$} &\multirow{2}{*}{DoF}& \multicolumn{3}{c|}{GMRES-2S} &\multicolumn{3}{c|}{GMRES-${\bf I}$}& \multicolumn{2}{c}{FAM-MG}  \\
					\cline{4-11}
					&&&$\mathrm{Iter}$&$\mathrm{CPU}$&${\rm RES}$&$\mathrm{Iter}$&$\mathrm{CPU}$&${\rm RES}$&$\mathrm{CPU}$&${\rm RES}$\\
					\hline
					\multirow{4}{*}{0.1} 
					&3375    &216000   &5   &0.42s   &5.92e-8 &50  &1.80s    &8.55e-8&2.98s   &5.05e-7 \\
					&29791   &1906624  &5   &3.62s   &7.43e-8 &111 &27.10s   &9.33e-8&14.44s  &4.46e-7 \\
					&250047  &16003008 &5   &27.55s  &7.15e-8 &323 &607.14s  &9.84e-8&123.89s &2.28e-7 \\
					&2048383 &131096512&5   &208.56s &5.73e-8 &1068&15308.26s&9.97e-8&1267.51s&3.18e-7 \\
					\hline
					\multirow{4}{*}{0.5} 
					&3375    &216000   &5   &0.40s   &8.11e-8 &62  &2.12s    &9.12e-8&2.85s   &2.50e-6 \\
					&29791   &1906624  &5   &3.49s   &9.16e-8 &131 &31.46s   &9.21e-8&13.72s  &2.11e-6 \\
					&250047  &16003008 &5   &27.39s  &7.80e-8 &347 &656.47s  &9.78e-8&119.88s &1.31e-6 \\
					&2048383 &131096512&5   &206.29s &5.97e-8 &1060&15215.04s&9.99e-8&1237.80s&5.24e-7 \\
					\hline
					\multirow{4}{*}{0.9}
					&3375    &216000   &5   &0.37s   &1.06e-7 &76  &2.55s    &9.54e-8&2.79s   &3.59e-6 \\
					&29791   &1906624  &5   &3.45s   &1.05e-7 &162 &38.97s   &9.98e-8&13.60s  &6.17e-6 \\
					&250047  &16003008 &5   &27.62s  &8.47e-8 &359 &685.34s  &9.74e-8&113.02s &1.79e-6 \\
					&2048383 &131096512&5   &209.76s &6.77e-8 &1066&15268.06s&9.92e-8&1204.87s&1.08e-6 \\
					\hline
				\end{tabular}
			\end{center}
		\end{table}
	}
\end{example}

\section{Concluding Remarks and Future Works}
In this paper, a novel two-sided PinT preconditioning technique for the all-at-once system from the non-local evolutionary equation with variable coefficients has been proposed. Theoretically, we have shown that the condition number of the two-sided preconditioned matrix is uniformly bounded by a constant independent of matrix size. Also, a fast implementation of Krylov subspace solver for the two-sided preconditioned system have been proposed. The proposed implementation is also a fast direct PinT solver for the unpreconditioned all-at-once system in the case of $a$ being a constant.
Numerical results reported have confirmed the effectiveness of the proposed preconditioning technique and consistency of the proposed theoretical analysis.

As the study of PinT fast \textcolor{black}{solvers} for the non-local evolutionary equation is still at its infancy, there are some extreme situations that can not be well handled by methods in the literature  as well as our proposed preconditioning technique. We list  the following issues as open problems 
\begin{description}
\item[$\bullet$] As shown in Theorem \ref{mainthm}, $\hat{a}/\check{a}$ is a uniform bound of the condition number of the preconditioned matrix, which means a small value of $\hat{a}/\check{a}$ guarantees a small condition number and thus a small iteration number of \textcolor{black}{the} Krylov subspace solvers. However, \textcolor{black}{if} $\hat{a}/\check{a}$ is large or $a$ even has zeros, then there is no guarantee that the condition number of the preconditioned matrix is small, which may lead to a slow convergence of \textcolor{black}{the} Krylov subspace solvers.
\item[$\bullet$] The fast implementation of the proposed preconditioning technique utilizes the fast diagonalizability of the constant Laplacian matrix ${\bf L}_1$. Such fast diagonalizability relies on the uniform spatial grid. When the physical domain $\Omega$ is irregular, there is no such uniform spatial grid discretization for $\Omega$. In such situation, our proposed preconditioning technique may not be applied directly.
\end{description}
Developing fast PinT \textcolor{black}{solvers} for the above introduced tough situations has to be the subject of future investigations.

\section*{Acknowledgements}
This research was supported by research grants HKRGC GRF 12306616, 12200317, 12300218 and 12300519, NSAF U930402.
and NSFC 11801479.

\bibliographystyle{plainnat}
\bibliography{myreferences}
\end{document}